\documentclass{amsart}
\usepackage{mathrsfs,latexsym,amsfonts,amssymb}
\newtheorem{theorem}{Theorem}[section]
\newtheorem{lemma}[theorem]{Lemma}
\newtheorem{corollary}[theorem]{Corollary}

\theoremstyle{definition}
\newtheorem{definition}[theorem]{Definition}
\newtheorem{proposition}[theorem]{Proposition}

\begin{document}

\title[The mapping $i_{2}$ on the free paratopological groups]
{The mapping $i_{2}$ on the free paratopological groups}

\author{Fucai Lin*}
\address{(Fucai Lin): School of mathematics and statistics,
Minnan Normal University, Zhangzhou 363000, P. R. China}
\email{linfucai2008@aliyun.com}

\author{Chuan Liu}
\address{(Chuan Liu): Department of Mathematics,
Ohio University Zanesville Campus, Zanesville, OH 43701, USA}
\email{liuc1@ohio.edu}

\thanks{The first author is supported by the NSFC (Nos.11201414, 11471153), the Natural Science Foundation of Fujian Province (No.2012J05013) of China, the Training Programme Foundation for Excellent Youth Researching Talents of Fujian¡¯s Universities (JA13190) and the foundation of The Education Department of Fujian Province(No. JA14200).}

\thanks{*corresponding author}


\begin{abstract}
Let $FP(X)$ be the free paratopological group over a topological space $X$. For each non-negative integer $n\in\mathbb{N}$, denote by $FP_{n}(X)$ the subset of $FP(X)$ consisting of all words of reduced length at most $n$, and $i_{n}$ by the natural mapping from $(X\bigoplus X^{-1}\bigoplus\{e\})^{n}$ to $FP_{n}(X)$. In this paper, we mainly improve some results of A.S. Elfard and P. Nickolas's [On the topology of free paratopological groups. II, Topology Appl.,  160(2013), 220--229.]. The main result is that the natural mapping $i_{2}: (X\bigoplus X_{d}^{-1}\bigoplus\{e\})^{2}\longrightarrow FP_{2}(X)$ is a closed mapping
if and only if every neighborhood $U$ of the diagonal $\Delta_{1}$ in $X_{d}\times X$ is a member of the finest quasi-uniformity on $X$, where $X$ is a $T_{1}$-space and $X_{d}$ denotes $X$ when equipped with the discrete topology in place of its given topology.
\end{abstract}

\maketitle

\section{Introduction}
In 1941, free topological groups were introduced by A.A. Markov in \cite{MA} with the clear idea of extending the well-known construction of a free group from group theory to topological groups. Now, free topological groups have become a powerful tool of study in the theory of topological groups and serve as a source of various examples and as an instrument for proving new theorems, see \cite{AT2008}.

As in free topological groups, S. Romaguera, M. Sanchis and M.G. Tkachenko in \cite{RS} defined free paratopological groups and proved the existence of the free paratopological group $FP(X)$ for every topological space $X$. Recently, A.S. Elfard, F.C. Lin, P. Nickolas and N.M. Pyrch have investigated some properties of free paratopological groups, see \cite{EN2012, EN2013, LFC2012, LF, PN2006, PN2011}.

For each non-negative integer $n\in\mathbb{N}$, denote by $FP_{n}(X)$ the subset of $FP(X)$ consisting of all words of reduced length at most $n$, and $i_{n}$ by the natural mapping from $(X\bigoplus X^{-1}\bigoplus\{e\})^{n}$ to $FP_{n}(X)$. In this paper, we mainly improve some results of A.S. Elfard and P. Nickolas's. The main result is that the natural mapping $i_{2}: (X\bigoplus X_{d}^{-1}\bigoplus\{e\})^{2}\longrightarrow FP_{2}(X)$ is a closed mapping
if and only if every neighborhood $U$ of the diagonal $\Delta_{1}$ in $X_{d}\times X$ is a member of the finest quasi-uniformity on $X$,  where $X$ is a $T_{1}$-space and $X_{d}$ denotes $X$ when equipped with the discrete topology in place of its given topology.

\section{Preliminaries}
All mappings are continuous. We denote by $\mathbb{N}$ and $\mathbb{Z}$ the sets of all natural
numbers and the integers, respectively. The letter $e$
denotes the neutral element of a group. Readers may consult
\cite{AT2008, E1989, Gr1984, F1982} for notations and terminology not
explicitly given here.

Recall that a {\it topological group} $G$ is a group $G$ with a
(Hausdorff) topology such that the product mapping of $G \times G$ into
$G$ is jointly continuous and the inverse mapping of $G$ onto itself
associating $x^{-1}$ with an arbitrary $x\in G$ is continuous. A {\it
paratopological group} $G$ is a group $G$ with a topology such that
the product mapping of $G \times G$ into $G$ is jointly continuous.

\begin{definition}\cite{RS}
Let $X$ be a subspace of a paratopological group $G$. Assume that
\begin{enumerate}
\item The set $X$ generates $G$ algebraically, that is $<X>=G$;

\item  Each continuous mapping $f: X\rightarrow H$ to a paratopological group $H$ extends to a continuous homomorphism $\hat{f}: G\rightarrow H$.
\end{enumerate}
Then $G$ is called the {\it Markov free paratopological group on} $X$ and is denoted by $FP(X)$.
\end{definition}

Again, if all the groups in the above definitions are Abelian, then we get the definition of the {\it Markov free Abelian paratopological group on} $X$ which will be denoted by $AP(X)$.

By \cite{RS}, $FPX$ and $AP(X)$ exist for every space $X$ and the underlying abstract groups of $FPX$ and $AP(X)$ are the free groups on the underlying set of the topological space $X$ respectively. We denote these abstract groups by $FP_{a}(X)$ and $AP_{a}(X)$ respectively.

Since $X$ generates the free group $FP_{a}(X)$, each element $g\in FP_{a}(X)$ has the form $g=x_{1}^{\varepsilon_{1}}\cdots x_{n}^{\varepsilon_{n}}$, where $x_{1}, \cdots, x_{n}\in X$ and $\varepsilon_{1}, \cdots, \varepsilon_{n}=\pm 1$. This word for $g$ is called {\it reduced} if it contains no pair of consecutive symbols of the form $xx^{-1}$ or $x^{-1}x$. It follow that if the word $g$ is reduced and non-empty, then it is different from the neutral element of $FP_{a}(X)$. For every non-negative integer $n$, denote by $FP_{n}(X)$ and $AP_{n}(X)$ the subspace of paratopological groups $FP(X)$ and $AP(X)$ that consists of all words of reduced length $\leq n$ with respect to the free basis $X$, respectively.

Let $X$ be a $T_{1}$-space. For each $n\in\mathbb{N}$, denote by $i_{n}$ the multiplication mapping from $(X\bigoplus X_{d}^{-1}\bigoplus\{e\})^{n}$ to $B_{n}(X)$, $i_{n}(y_{1}, \cdots, y_{n})=y_{1}\cdot \cdots\cdot y_{n}$ for every point $(y_{1}, \cdots, y_{n})\in (X\bigoplus X_{d}^{-1}\bigoplus\{e\})^{n}$, where $X_{d}^{-1}$ denotes the set $X^{-1}$ equipped with the discrete topology and $B_{n}(X)$ denotes $FP_{n}(X)$ or $AP_{n}(X)$.

By a {\it quasi-uniform space} $(X, \mathscr{U})$ we mean the natural analog of a {\it uniform space} obtained by dropping the symmetry axiom. For each quasi-uniformity $\mathscr{U}$ the filter  $\mathscr{U}^{-1}$ consisting of the inverse relations $U^{-1}=\{(y, x): (x, y)\in U\}$ where $U\in\mathscr{U}$ is called the {\it conjugate quasi-uniformity} of $\mathscr{U}$.

Let $X$ be a topological space. Then $X_{d}$ denotes $X$ when equipped with the discrete topology in place of its given topology. We denote the diagonals of $X_{d}\times X$ and $X\times X_{d}$ by $\Delta_{1}$ and $\Delta_{2}$, respectively.
In \cite{PN2006}, the authors proved that $X^{-1}$ is discrete in free paratopological group $FP(X)$ and $AP(X)$ over $X$ if $X$ is a $T_{1}$-space. We denote the sets $\{(x^{-1}, y): (x, y)\in X\times X\}$ and $\{(x, y^{-1}): (x, y)\in X\times X\}$ by $\Delta_{1}^{\ast}$ and $\Delta_{2}^{\ast}$, respectively.

\section{Main results}
\begin{theorem}\label{t0}\cite{EN2013}
If $X$ is a $T_{1}$-space, then the mapping $$i_{2}\mid _{i_{2}^{-1}(FP_{2}(X)\setminus FP_{1}(X))}: i_{2}^{-1}(FP_{2}(X)\setminus FP_{1}(X))\longrightarrow FP_{2}(X)\setminus FP_{1}(X)$$ is a homeomorphism.
\end{theorem}

\begin{theorem}\label{t9}\cite{EN2012}
Let $X$ be a $T_{1}$-space and let $w=x_{1}^{\epsilon_{1}}x_{2}^{\epsilon_{2}}\cdots x_{n}^{\epsilon_{n}}$ be a reduced word in $FP_{n}(X)$, where $x_{i}\in X$ and $\epsilon_{i}=\pm 1$, for all $i=1, 2, \cdots, n$, and if $x_{i}=x_{i+1}$ for some $i=1, 2, \cdots, n-1$, then $\epsilon_{i}=\epsilon_{i+1}$. Then the collection $\mathscr{B}$ of all sets of the form $U_{1}^{\epsilon_{1}}U_{2}^{\epsilon_{2}}\cdots U_{n}^{\epsilon_{n}}$, where, for all $i=1, 2, \cdots, n$, the set $U_{i}$ is a neighborhood of $x_{i}$ in $X$ when $\epsilon_{i}=1$ and $U_{i}=\{x_{i}\}$ when $\epsilon_{i}=-1$ is a base for the neighborhood system at $w$ in $FP_{n}(X)$.
\end{theorem}

\begin{theorem}\label{t2}\cite{EN2012}
Let $X$ be a $T_{1}$-space and let $w=\epsilon_{1}x_{1}+\epsilon_{2}x_{2}+\cdots+\epsilon_{n}x_{n}$ be a reduced word in $AP_{n}(X)$, where $x_{i}\in X$ and $\epsilon_{i}=\pm 1$, for all $i=1, 2, \cdots, n$, and if $x_{i}=x_{j}$ for some $i, j=1, 2, \cdots, n$, then $\epsilon_{i}=\epsilon_{j}$. Then the collection $\mathscr{B}$ of all sets of the form $\epsilon_{1}U_{1}+\epsilon_{2}U_{2}+\cdots+\epsilon_{n}U_{n}$, where, for all $i=1, 2, \cdots, n$, the set $U_{i}$ is a neighborhood of $x_{i}$ in $X$ when $\epsilon_{i}=1$ and $U_{i}=\{x_{i}\}$ when $\epsilon_{i}=-1$ is a base for the neighborhood system at $w$ in $AP_{n}(X)$.
\end{theorem}

\begin{theorem}\label{t3}
If $X$ is a $T_{1}$-space, then the mapping $$f=i_{2}\mid _{i_{2}^{-1}(AP_{2}(X)\setminus AP_{1}(X))}: i_{2}^{-1}(AP_{2}(X)\setminus AP_{1}(X))\longrightarrow AP_{2}(X)\setminus AP_{1}(X)$$ is a 2 to 1, open and perfect mapping.
\end{theorem}

\begin{proof}
Obviously, $f$ is a 2 to 1 mapping. Next, we shall prove that $f$ is open and closed. Let $C_{2}(X)=AP_{2}(X)\setminus AP_{1}(X)$ and $C_{2}^{\ast}(X)=i_{2}^{-1}(AP_{2}(X)\setminus AP_{1}(X))$. Obviously, we have $$C_{2}^{\ast}(X)=(X\times X)\bigoplus (X_{d}^{-1}\times X_{d}^{-1})\bigoplus (X_{d}^{-1}\times X)\setminus\Delta_{1}^{\ast}\bigoplus (X\times X_{d}^{-1})\setminus\Delta_{2}^{\ast}.$$

(1) The mapping $f$ is open.

Let $(x_{1}^{\epsilon_{1}}, x_{2}^{\epsilon_{2}})\in C_{2}^{\ast}(X)$, where $x_{1}, x_{2}\in X$ and $x_{1}\neq x_{2}$ if $\epsilon_{1}\neq\epsilon_{2}$. Let $U$ be a neighborhood of $(x_{1}^{\epsilon_{1}}, x_{2}^{\epsilon_{2}})$ in $C_{2}^{\ast}(X)$. By Theorem~\ref{t2}, $f(U)$ is a neighborhood of $x_{1}^{\epsilon_{1}} x_{2}^{\epsilon_{2}}$ in $C_{2}(X)$. (Indeed, the argument is similar to the proof of \cite[Theorem 3.4]{EN2013}.) Therefore, $f$ is open.

(2) The mapping $f$ is closed.

Let $E$ be a closed subset of $C_{2}^{\ast}(X)$. To show that $i_{2}(E)$ is closed in $C_{2}(X)$ take $w\in \overline{i_{2}(E)}$. Next, we shall show that $w\in i_{2}(E)$. Indeed, it is obvious that $w$ has a reduced form $w=\epsilon_{1}x_{1}+\epsilon_{2}x_{2}$, where $\epsilon_{i}=1$ or -1 ($i=1, 2$), $x_{1}, x_{2}\in X$ and $x_{1}\neq x_{2}$ if $\epsilon_{1}\neq\epsilon_{2}$.

Suppose that $w=x+y\notin i_2(E)$, where $x=\epsilon_{1}x_{1}$ and $y=\epsilon_{2}x_{2}$. Then $\{(x, y), (y, x)\}\cap E=\emptyset$. Since $E$ is closed, we can pick open neighborhoods $V(x)$ of $x$ in $X\cup X_{d}^{-1}$, $V(y)$ of $y$ in $X\cup X_{d}^{-1}$ such that $(V(x)\times V(y))\cap E=\emptyset$, $(V(y)\times V(x))\cap E=\emptyset$.  Let $U=(V(x)\times V(y))\cup (V(y)\times V(x))$. Then $U$ is open. Since $f$ is an open map, we have $f(U)$ is a neighborhood of $w$ and $f(U)\cap i_2(E)=\emptyset$. This contradicts with $w\in \overline{i_2(E)}$.

\end{proof}

For arbitrary space $X$, the mapping $f: X\longrightarrow \mathbb{Z}$ defined by setting $f(x)=1$ for all $x\in X$ is continuous, and thus extends to a continuous homomorphism $\widehat{f}: AP(X)\longrightarrow \mathbb{Z}$. Therefore, the collection of sets $Z_{n}(X)=\widehat{f}^{-1}(\{n\})$ for $n\in \mathbb{Z}$ forms a partition of $AP(X)$ into clopen subspaces.

For a $T_{1}$-space, define $$g: (X_{d}\times X)\bigoplus (X\times X_{d})\bigoplus (\{e\}\times\{e\})\longrightarrow AP_{2}(X)\cap Z_{0}(X)$$by
\[g(x, y)=\left\{
\begin{array}{lll}
-x+y, & \mbox{if}\ (x, y)\in X_{d}\times X;\\
x-y, & \mbox{if}\ (x, y)\in X\times X_{d};\\
e, & \mbox{if}\ x=y.\end{array}\right.\]

Let $g_{j}=i_{2}\mid _{i_{2}^{-1}(AP_{2}(X)\cap Z_{j}(X))}$ for $j=-2, \cdots, 2$, where $i_{2}: (X\bigoplus X_{d}^{-1}\bigoplus\{e\})^{2}\longrightarrow AP_{2}(X)$. Obviously, $i_{2}=\bigoplus_{j=-2}^{j=2}\{g_{j}\}$, and $i_{2}$ is a closed (resp. quotient) mapping if and only if each $g_{j}$ is a closed (resp. quotient) mapping, where  $j=-2, \cdots, 2$. By Theorem~\ref{t3}, it is easy to see that $g_{-2}$ and $g_{2}$ are open and closed. Moreover, since $-X$ occurs with the discrete topology and $X$ occurs with its original topology in $AP(X)$, the mappings $g_{-1}$ and $g_{1}$ are open and closed. Obviously, $g$ is a closed (resp. quotient) mapping if and only if $g_{0}$ is a closed (resp. quotient) mapping. Therefore, we have the following result:

\begin{lemma}\label{l0}
Let $X$ be a $T_{1}$-space. Then $i_{2}$ is a closed (resp. quotient) mapping if and only if $g$ is a closed (resp. quotient) mapping.
\end{lemma}

\begin{lemma}\label{l1}\cite{EN2013}
Let $X$ be a space and let $\Delta_{1}$ be the diagonal in the space $X_{d}\times X$. Then $\Delta_{1}$ is closed if and only if $X$ is $T_{1}$. Similarly for the diagonal $\Delta_{2}$ in the space $X\times X_{d}$.
\end{lemma}

Suppose that $\mathscr{U}^{\ast}$ is the finest quasi-uniformity of a space $X$. We say that $P=\{U_{i}\}_{i\in\mathbb{N}}$ is a sequence of $\mathscr{U}^{\ast}$ if each $U_{i}\in \mathscr{U}^{\ast}$. Put $$^{\omega}\mathscr{U}^{\ast}=\{P: P\ \mbox{is a sequence of}\ \mathscr{U}^{\ast}\}.$$

For each $n\in\mathbb{N}$ and $P=\{U_{i}\}_{i\in\mathbb{N}}\in\ ^{\omega}\mathscr{U}^{\ast}$, let

$\mathscr{Q}_{n}(\mathbb{N})=\{A\subset \mathbb{N}: |A|=n\},$

$W_{n}(P)=\{-x_{1}+y_{1}-\cdots -x_{n}+y_{n}: (x_{j}, y_{j})\in U_{i_{j}}\ \mbox{for}\ j=1, 2, \cdots, n,$ $\{i_{1}, i_{2}, \cdots, i_{n}\}\in \mathscr{Q}_{n}(\mathbb{N})\},$ and

$\mathscr{W}_{n}=\{W_{n}(P): P\in\ ^{\omega}\mathscr{U}^{\ast}\}$.

{\bf Remark} In the above definition, for $P=\{U_{i}\}_{i\in\mathbb{N}}\in\ ^{\omega}\mathscr{U}^{\ast}$, there may exist $i\neq j$ such that $U_{i}=U_{j}$. In particular, for every $U\in\mathscr{U}^{\ast}$, we have $\{U_{i}=U\}_{i\in\mathbb{N}}$ is also in $^{\omega}\mathscr{U}^{\ast}$. Moreover, the reader should note that the representation of elements of $W_{n}(P)$ need not be a reduced representation.

\begin{theorem}\cite{LFC2012}\label{t8}
For each $n\in \mathbb{N}$, the family $\mathscr{W}_{n}$ is a neighborhood base of $e$ in $AP_{2n}(X)$.
\end{theorem}

The proof of the following Theorem is a modification of \cite[Theorem 3.10]{EN2013}.

\begin{theorem}\label{t7}
Let $X$ be a $T_{1}$-space. Then the mapping $$i_{2}: (X\bigoplus X_{d}^{-1}\bigoplus\{e\})^{2}\longrightarrow AP_{2}(X)$$ is a quotient mapping if and only if every neighborhood $U$ of the diagonal $\Delta_{1}$ in $X_{d}\times X$ is a member of the finest quasi-uniformity $\mathscr{U}^{\ast}$ on $X$.
\end{theorem}

\begin{proof}
Put $Z=(X_{d}\times X)\bigoplus (X\times X_{d})\bigoplus (\{e\}\times\{e\})$.

{\bf Necessity.} Suppose that $i_{2}$ is a quotient mapping.  It follows from Lemma~\ref{l0} that $g: Z\longrightarrow AP_{2}(X)\cap Z_{0}(X)$ is a quotient mapping. Let $U$ be a neighborhood of $\Delta_{1}$ in $X_{d}\times X$. Obviously, $U\cup (-U)$ is a neighborhood of $\Delta_{1}\cup\Delta_{2}$ in $Z$. Let $P=\{U_n\}_{n\in\mathbb{N}}$, where $U_{n}=U$ for each $n\in \mathbb{N}$. Let  $W_{1}(P)=\{-x+y: (x, y)\in U\}$. Then $g^{-1}(W_{1}(P))=U\cup (-U)\cup \{(e, e)\}$ that is a neighborhood of $\Delta_{1}\cup\Delta_{2}\cup \{(e, e)\}$ in $Z$, then $W_{1}(P)$ is a neighborhood of $e$ in $AP_{2}(X)\cap Z_{0}(X)$, and hence in $AP_{2}(X)$. By Theorem~\ref{t8}, there exists $Q\in\ ^{\omega}\mathscr{U}^{\ast}$  such that $W_{1}(Q)\subset W_{1}(P)$, where $Q=\{V_n\}_{n\in\mathbb{N}}$. Then $V_1\subset U$, hence $U\in\mathscr{U}^{\ast}$.

{\bf Sufficiency.} Suppose that every neighborhood $U$ of the diagonal $\Delta_{1}$ in $X_{d}\times X$ is a member of the finest quasi-uniformity $\mathscr{U}^{\ast}$ on $X$. To show that $i_{2}$ is a quotient mapping, it follows from Lemma~\ref{l0} that it suffices to show that the mapping $g: Z\longrightarrow AP_{2}(X)\cap Z_{0}(X)$ is a quotient mapping. Take a subset $A\subset AP_{2}(X)\cap Z_{0}(X)$ such that $g^{-1}(A)$ is open in $Z$. Put $U=g^{-1}(A)\cap (X_{d}\times X)$ and $V=g^{-1}(A)\cap (X\times X_{d})$. Firstly, we show the following claim:

{\bf Claim:} If $e\not\in A$, then $A$ is open in $AP_{2}(X)\cap Z_{0}(X)$.

Indeed, since $e\not\in A$, $U\cap \Delta_{1}=\emptyset$ and $V\cap \Delta_{2}=\emptyset$. By Lemma~\ref{l1}, $\Delta_{1}$ and $\Delta_{2}$ are closed in $X_{d}\times X$ and $X\times X_{d}$, respectively, and $X_{d}\times X\setminus\Delta_{1}$ and $X\times X_{d}\setminus\Delta_{2}$ are open in $X_{d}\times X$ and $X\times X_{d}$, respectively. Hence $U\cup V$ is open in the space
$i_{2}^{-1}(AP_{2}(X)\setminus AP_{1}(X))$, and by Theorem~\ref{t3}, $g(U\cup V)=A$ is open in $AP_{2}(X)\cap Z_{0}(X)$.

Next we shall show that $A$ is open in $AP_{2}(X)\cap Z_{0}(X)$. Take arbitrary $a\in A$. Then it suffices to show that $A$ is open neighborhood of $a$.

{\bf Case 1:} $a=e$.

Obviously, $U$ and $V$ are open neighborhoods of $\Delta_{1}$ and $\Delta_{2}$ in $X_{d}\times X$ and $X\times X_{d}$, respectively. Therefore, $S=U\cap (V^{-1})$ is an open neighborhood of $\Delta_{1}$ in $X_{d}\times X$, and thus $S\in \mathscr{U}^{\ast}$. Let $W_{1}(R)=\{-x+y: (x, y)\in S\}$, where $R=\{S_{n}\}_{n\in\mathbb{N}}$ and $S_{n}=S$ for each $n\in\mathbb{N}$. By Theorem~\ref{t8}, $W_{1}(R)$ is a neighborhood of $e$ in $AP_{2}(X)$. Since $S=U\cap (V^{-1})$ and the definition of $g$, it is easy to see that $W_{1}(R)\subset A$. Therefore, $A$ is a neighborhood of $e$ in $AP_{2}(X)$, hence in $AP_{2}(X)\cap Z_{0}(X).$

{\bf Case 2:} $a\neq e$.

Let $W$ be an open neighborhood of $a$ in $AP_{2}(X)\cap Z_{0}(X)$ such that $e\not\in W$. Then the set $g^{-1}(A\cap W)$ is open in $Z$, and it follows from Claim that $A\cap W$ is an open neighborhood of $a$ in $AP_{2}(X)\cap Z_{0}(X).$ Hence $A$ is open in $AP_{2}(X)\cap Z_{0}(X).$
\end{proof}

The next theorem is the main result in \cite{EN2013}, and some related concepts can be seen in \cite{F1982}. Next, we shall improve this result in Theorem~\ref{t5}.

\begin{theorem}\label{t4}\cite{EN2013}
Let $X$ be a $T_{1}$-space. Then the followings are equivalent:
\begin{enumerate}
\item The mapping $i_{2}: (X\bigoplus X_{d}^{-1}\bigoplus\{e\})^{2}\longrightarrow FP_{2}(X)$ is a quotient mapping;

\item  Every neighborhood $U$ of the diagonal $\Delta_{1}$ in $X_{d}\times X$ is a member of the finest quasi-uniformity $\mathscr{U}^{\ast}$ on $X$;

\item Every neighbornet of $X$ is normal;

\item The finest quasi-uniformity $\mathscr{U}^{\ast}$ on $X$ consists of all neighborhoods of the diagonal $\Delta_{1}$ in $X_{d}\times X$;

\item If $N_{x}$ is a neighborhood of $x$ for all $x\in X$, then there exists a neighborhood $M_{x}$ of $x$ such that $\bigcup_{y\in M_{x}}M_{y}\subset N_{x}$ for all $x\in X$;

\item If $N_{x}$ is a neighborhood of $x$ for all $x\in X$, then there exists a quasi-pseudometric $d$ on $X$ such that $d_{x}$ is upper semi-continuous and $B_{d}(x, 1)\subset N_{x}$ for all $x\in X$.
\end{enumerate}
\end{theorem}

Let $X$ be a set. Define $j_{2}, k_{2}: X\times X\longrightarrow F_{a}(X)$ by $j_{2}(x, y)=x^{-1}y$ and $k_{2}(x, y)=yx^{-1}$.

\begin{theorem}\label{t6}\cite{EN2013}
Let $X$ be a topological space. Then the collection $\mathscr{B}$ of sets $j_{2}(U)\cup k_{2}(U)$ for $U\in\mathscr{U}^{\ast}$ is a base of neighborhoods at the identity $e$ in $FP_{2}(X)$.
\end{theorem}

Now we can prove the main theorem in this paper.

\begin{theorem}\label{t5}
Let $X$ be a $T_{1}$-space. Then the following are equivalent:
\begin{enumerate}
\item The mapping $i_{2}: (X\bigoplus X_{d}^{-1}\bigoplus\{e\})^{2}\longrightarrow FP_{2}(X)$ is a quotient mapping;

\item The mapping $i_{2}: (X\bigoplus X_{d}^{-1}\bigoplus\{e\})^{2}\longrightarrow AP_{2}(X)$ is a quotient mapping;

\item The mapping $i_{2}: (X\bigoplus X_{d}^{-1}\bigoplus\{e\})^{2}\longrightarrow FP_{2}(X)$ is a closed mapping;

\item The mapping $i_{2}: (X\bigoplus X_{d}^{-1}\bigoplus\{e\})^{2}\longrightarrow AP_{2}(X)$ is a closed mapping.
\end{enumerate}
\end{theorem}

\begin{proof}
Obviously, we have (3) $\Rightarrow$ (1) and (4) $\Rightarrow$ (2). Moreover, it follows from Theorems~\ref{t7} and ~\ref{t4} that we have (2) $\Rightarrow$ (1). It suffices to show that (1) $\Rightarrow$ (3) and (2) $\Rightarrow$ (4).

{\bf (1) $\Rightarrow$ (3)}. Clearly,  both $FP_{2}(X)\setminus FP_{1}(X)$ and $FP_{1}(X)\setminus\{e\}$ are open in $FP_{2}(X)$. Let $E$ be a closed subset in $(X\bigoplus X_{d}^{-1}\bigoplus\{e\})^{2}$. To show that $i_{2}(E)$ is closed in $FP_{2}(X)$ take $w\in \overline{i_{2}(E)}$.

{\bf Case a1:} $w\in FP_{1}(X)\setminus\{e\}$.

Suppose $w\notin i_2(E)$, then $(w, e)\notin E\ \mbox{and}\ (e, w)\notin E$. Since $E$ is closed, there is open neighborhood $U$ (open in $X\cup X_{d}^{-1}$) of $w$ such that $(U\times \{e\})\cap E= \emptyset$ and $(\{e\}\times U)\cap E= \emptyset$. Obviously, we have $(U\times \{e\})\cup (\{e\}\times U)=i_2^{-1}(U)$. Then $U$ is open in $FP_2(X)$ since $(U\times \{e\})\cup (\{e\}\times U)$ is open in $(X\bigoplus X_{d}^{-1}\bigoplus\{e\})^{2}$ and $i_2$ is a quotient map. Hence $U\cap i_2(E)=\emptyset$, which contradicts $w\in \overline{i_2(E)}$.

{\bf Case a2:} $w\in FP_{2}(X)\setminus FP_{1}(X)$.

Let $w=w_{1}^{\epsilon_{1}}w_{2}^{\epsilon_{2}}$, where $w_{i}\in X$ and $\epsilon_{i}=1$ or -1 ($i=1, 2$). Suppose that $w\not\in i_2(E)$. Then $(w_{1}^{\epsilon_{1}}, w_{2}^{\epsilon_{2}})\not\in E$.

{\bf Subcase a21:} $\epsilon_{1}=\epsilon_{2}=1$.

Since $(w_{1}, w_{2})\not\in E$ and $E$ is closed in $(X\bigoplus X_{d}^{-1}\bigoplus\{e\})^{2}$,  there exist neighborhoods $U$ and $V$ of $w_{1}$ and $w_{2}$ in $X$, respectively, such that
$(U\times V)\cap E=\emptyset$. Therefore, it is easy to see that $UV\cap i_2(E)=\emptyset$. From Theorem~\ref{t9} it follows that $UV$ is a neighborhood of $w$, hence $w\not\in\overline{i_2(E)}$, which is a contradiction.

{\bf Subcase a22:} $\epsilon_{1}=\epsilon_{2}=-1$.

From Theorem~\ref{t9} it follows that $\{w_{1}^{-1}w_{2}^{-1}\}$ is a neighborhood of $w$, then $w\not\in\overline{i_2(E)}$, which is a contradiction.

{\bf Subcase a23:} $\epsilon_{1}\neq\epsilon_{2}$.

Without loss of generality, we may assume that $\epsilon_{1}=1$ and $\epsilon_{2}=-1$. Then since $(w_{1}, w_{2}^{-1})\not\in E$ and $E$ is closed in $(X\bigoplus X_{d}^{-1}\bigoplus\{e\})^{2}$,  there exists a neighborhood $U$ of $w_{1}$ in $X$ such that
$(U\times \{w_{2}^{-1}\})\cap E=\emptyset$ and $w_{2}\not\in U$. (This is possible since $X$ is $T_{1}$.) Obviously, $U w_{2}^{-1}\subset FP_{2}(X)\setminus FP_{1}(X)$. Therefore, it is easy to see that $U w_{2}^{-1}\cap i_2(E)=\emptyset$. From Theorem~\ref{t9} it follows that $U w_{2}^{-1}$ is a neighborhood of $w$, hence $w\not\in\overline{i_2(E)}$, which is a contradiction.

Therefore, we have $w\in i_2(E)$.

{\bf Case a3:} $w=e$.

Suppose that $e\not\in i_{2}(E)$. Then $E\cap (\Delta_{1}\cup \Delta_{2}\cup\{(e, e)\})=\emptyset$. For any $x\in X$, since $E$ does not contain points $(x^{-1}, x)$ and $(x, x^{-1})$, there exists an open neighborhood $U(x)$ of $x$ in $X$ such that $(\{x^{-1}\}\times U(x))\cap E=\emptyset$ and $(U(x)\times \{x^{-1}\})\cap E=\emptyset$. Let $U=\bigcup_{x\in X}(\{x^{-1}\}\times U(x))$ and $V=\bigcup_{x\in X}(U(x)\times \{x^{-1}\})$. Then $U\cap E=\emptyset$ and $V\cap E=\emptyset$. Let $W=U\cup V\cup\{e\}\times\{e\}$. Then $W$ is open in $(X\bigoplus X_{d}^{-1}\bigoplus\{e\})^{2}$ by (2) of Theorem~\ref{t4}. Obviously, we have $W\cap E=\emptyset$. It is easy to see that $i_2^{-1}(i_2(W))=W$, then $i_2(W)$ is open since $i_2$ is a quotient map. Hence $i_2(W)\cap i_2(E)=\emptyset$, this is a contradiction.

\smallskip

{\bf (2) $\Rightarrow$ (4)}. ({\bf Note}: The proof is almost similar to (1) $\Rightarrow$ (3). However, we give out the proof for the convenience for readers.) Clearly, both $AP_{2}(X)\setminus AP_{1}(X)$ and $AP_{1}(X)\setminus\{e\}$ are open in $AP_{2}(X)$. Let $E$ be a closed subset in $(X\bigoplus -X_{d}\bigoplus\{e\})^{2}$. To show that $i_{2}(E)$ is closed in $AP_{2}(X)$ take $w\in \overline{i_{2}(E)}$.

{\bf Case b1:} $w\in AP_{1}(X)\setminus\{e\}$.

Suppose $w\notin i_2(E)$, then $(w, e)\notin E\ \mbox{and}\ (e, w)\notin E$. Since $E$ is closed, there is open neighborhood $U$ (open in $X\cup -X_{d}$) of $w$ such that $(U\times \{e\})\cap E= \emptyset$ and $(\{e\}\times U)\cap E= \emptyset$. Obviously, we have $(U\times \{e\})\cup (\{e\}\times U)=i_2^{-1}(U)$. Then $U$ is open in $AP_2(X)$ since $(U\times \{e\})\cup (\{e\}\times U)$ is open in $(X\bigoplus -X_{d}\bigoplus\{e\})^{2}$ and $i_2$ is a quotient map by Theorems ~\ref{t7} and~\ref{t4}. Then $U\cap i_2(E)=\emptyset$, that contradicts $w\in \overline{i_2(E)}$.

{\bf Case b2:} $w\in AP_{2}(X)\setminus AP_{1}(X)$.

Let $w=\epsilon_{1}w_{1}+\epsilon_{2}w_{2}$, where $w_{i}\in X$ and $\epsilon_{i}=1$ or -1 ($i=1, 2$). Suppose that $w\not\in i_2(E)$. Then $(\epsilon_{1}w_{1}, \epsilon_{2}w_{2})\not\in E$ and $(\epsilon_{2}w_{2}, \epsilon_{1}w_{1})\not\in E$.

{\bf Subcase b21:} $\epsilon_{1}=\epsilon_{2}=1$.

Since $\{(w_{1}, w_{2}), (w_{2}, w_{1})\}\not\in E$ and $E$ is closed in $(X\bigoplus -X_{d}\bigoplus\{e\})^{2}$,  there exist neighborhoods $U$ and $V$ of $w_{1}$ and $w_{2}$ in $X$, respectively, such that
$(U\times V\cup V\times U)\cap E=\emptyset$. Therefore, it is easy to see that $(U+V)\cap i_2(E)=\emptyset$. From Theorem~\ref{t2} it follows that $U+V$ is a neighborhood of $w$, hence $w\not\in\overline{i_2(E)}$, which is a contradiction.

{\bf Subcase b22:} $\epsilon_{1}=\epsilon_{2}=-1$.

From Theorem~\ref{t9} it follows that $\{-w_{1}-w_{2}\}$ is a neighborhood of $w$, then $w\not\in\overline{i_2(E)}$, which is a contradiction.

{\bf Subcase b23:} $\epsilon_{1}\neq\epsilon_{2}$.

Without loss of generality, we may assume that $\epsilon_{1}=1$ and $\epsilon_{2}=-1$. Then since $$\{(w_{1}, -w_{2}), (-w_{2}, w_{1})\}\not\in E$$ and $E$ is closed in $(X\bigoplus -X_{d}\bigoplus\{e\})^{2}$,  there exists a neighborhood $U$ of $w_{1}$ in $X$ such that
$(U\times \{w_{2}^{-1}\}\cup \{w_{2}^{-1}\}\times U)\cap E=\emptyset$ and $w_{2}\not\in U$. (This is possible since $X$ is $T_{1}$.) Obviously, $U-w_{2}\subset AP_{2}(X)\setminus AP_{1}(X)$. Therefore, it is easy to see that $(U-w_{2})\cap i_2(E)=\emptyset$. From Theorem~\ref{t2} it follows that $U-w_{2}$ is a neighborhood of $w$, hence $w\not\in\overline{i_2(E)}$, which is a contradiction.

Therefore, we have $w\in i_2(E)$.

{\bf Case b3:} $w=e$.

Suppose that $e\not\in i_{2}(E)$. Then $E\cap (\Delta_{1}\cup \Delta_{2}\cup\{(e, e)\})=\emptyset$. For any $x\in X$, since $E$ does not contain points $(-x, x)$ and $(x, -x)$, there exists an open neighborhood $U(x)$ of $x$ in $X$ such that $(\{-x\}\times U(x))\cap E=\emptyset$ and $(U(x)\times \{-x\})\cap E=\emptyset$. Let $U=\bigcup_{x\in X}(\{-x\}\times U(x))$ and $V=\bigcup_{x\in X}(U(x)\times \{-x\})$. Then $U\cap E=\emptyset$ and $V\cap E=\emptyset$. Let $W=U\cup V\cup\{e\}\times\{e\}$. Then $W$ is open in $(X\bigoplus -X_{d}\bigoplus\{e\})^{2}$ by Theorem~\ref{t4}. Obviously, we have $W\cap E=\emptyset$. It is easy to see that $i_2^{-1}(i_2(W))=W$, then $i_2(W)$ is open in $AP_2(X)$ since $i_2$ is a quotient map by Theorems ~\ref{t7} and~\ref{t4}. Hence $i_2(W)\cap i_2(E)=\emptyset$, which is a contradiction.

\end{proof}

\begin{proposition}
Let $X$ be a $T_1$-space. Then, for some $n\geq 3$, $i_n:  (X\bigoplus X_{d}^{-1}\bigoplus\{e\})^{n}\rightarrow FP_{n}(X)$ is a closed map if and only if $X$ is discrete.
\end{proposition}

\begin{proof}
If $X$ is discrete, then $FP(X)$ is discrete, it is easy to see that each $i_n$ is a closed map.

\smallskip

Let $i_n$ be a closed map for some $n\geq 3$. Since $X$ is $T_1$, then $X^{-1}$ is discrete. Suppose that $X$ is not discrete, then there exists $x\in X$ such that $x\in \overline {X\setminus \{x\}}$.
Let $$A=\{(x_\alpha, x_\alpha, x_\alpha^{-1}, e, \cdots, e)\in(X\bigoplus X_{d}^{-1}\bigoplus\{e\})^n: x_\alpha \in X\setminus \{x\}\}.$$ Then $A$ is a closed discrete subset of $(X\bigoplus X_{d}^{-1}\bigoplus\{e\})^n$, and therefore, $i_n(A)=X\setminus \{x\}$ is closed discrete subset, which is a contradiction. Hence $X$ is discrete.
\end{proof}

{\bf Note} Therefore, we can improve all results in \cite[Sections 4 and 5]{EN2013} from quotient mappings to closed mappings. For example, we have the following proposition.

\begin{proposition}\label{p0}
The mapping $i_{2}$ is a closed mapping for any countable $T_1$-space. In particular, the mapping $i_{2}$ is a closed mapping for any countable subspace of real line $\mathbb{R}$.
\end{proposition}

\begin{corollary}
$FP_{2}(\mathbb{Q})$ and $AP_{2}(\mathbb{Q})$ are Fr$\acute{e}$chet, where $\mathbb{Q}$ is the rational number of real line $\mathbb{R}$.
\end{corollary}

\begin{proof}
By Proposition~\ref{p0}, $i_{2}$ is a closed mapping. Then $FP_{2}(\mathbb{Q})$ and $AP_{2}(\mathbb{Q})$ are Fr$\acute{e}$chet since $(X\bigoplus X_{d}^{-1}\bigoplus\{e\})^{2}$ is Fr$\acute{e}$chet and closed mappings preserve the property of Fr$\acute{e}$chet.
\end{proof}

By \cite[Proposition 6.26]{F1982}, we also have the following proposition.

\begin{proposition}\label{p1}
For arbitrary compact first-countable Hausdorff space $X$, the mapping $i_{2}$ is closed if and only if $X$ is countable.
\end{proposition}

{\bf Acknowledgements}. We wish to thank
the reviewers for the detailed list of corrections, suggestions to the paper, and all her/his efforts
in order to improve the paper.


\begin{thebibliography}{99}
\bibitem{AT2008} A.V. Arhangel'ski\v{\i}, M. Tkachenko,
{\it Topological Groups and Related Structures}, Atlantis Press and
World Sci., Paris, 2008.

\bibitem{EN2012} A. S. Elfard, P. Nickolas,  {\it On the topology of free paratopological groups}, Bull. London Math. Soc., {\bf 44(6)} (2012), 1103--1115.

\bibitem{EN2013} A. S. Elfard, P. Nickolas,  {\it On the topology of free paratopological groups. II}, Topology Appl., {\bf 160}(2013), 220--229.

\bibitem{E1989} R. Engelking, {\it General Topology} (revised and completed edition), Heldermann
Verlag, Berlin, 1989.

\bibitem{F1982} P. Fletcher, W.F. Lindgren, {\it Quasi-uniform spaces}, Marcel Dekker, New York, 1982.

\bibitem{Gr1984} G. Gruenhage, {\it Generalized metric spaces}, K. Kunen, J.E. Vaughan eds.,
Handbook of Set-Theoretic Topology, North-Holland, (1984), 423--501.

\bibitem{LFC2012} F. Lin, {\it A note on free paratopological groups}, Topology Appl., {\bf 159}(2012), 3596--3604.

\bibitem{LF} F. Lin,  {\it Topological monomorphism between free paratopological groups},
Bulletin of the Belgian Mathematical Society-Simon Stevin, {\bf 19} (2012), 507--521.

\bibitem{MA} A.A. Markov, {\it On free topological groups},
Dokl. Akad. Nauk. SSSR 31(1941) 299--301.

\bibitem{PN2006} N.M. Pyrch, A.V. Ravsky, {\it On free paratopological groups},
Matematychni Studii 25(2006) 115--125.

\bibitem{PN2011} N.M. Pyrch, {\it Free paratopological groups and free products of paratopological groups}, Journal of Mathematical Sciences 174(2)(2011) 190--195.

\bibitem{RS} S. Romaguera, M. Sanchis, M.G. Tkachenko, {\it Free paratopological groups},
Topology Proceedings 27(2002) 1--28.
\end{thebibliography}
\end{document}